\documentclass[12pt,reqno]{amsart}
\usepackage{mathtools}
\usepackage{amssymb, mathabx, enumitem, mathrsfs}
\usepackage{color}
\usepackage[colorlinks=true,linkcolor=blue,urlcolor=blue,citecolor=blue, pagebackref]{hyperref}
\usepackage[alphabetic]{amsrefs}
\usepackage{tikz-cd}
\usepackage{tikz}
\usepackage{tikz-3dplot}
\usetikzlibrary{arrows}
\usepackage{graphics}
\usepackage{arydshln}

\usepackage[margin=1in]{geometry}

\newcommand{\C}{\mathbb{C}}

\newcommand{\Z}{\mathbb{Z}}

\newcommand{\op}{\operatorname}

\newtheorem{theorem}{Theorem}[section]

\newtheorem{corollary}[theorem]{Corollary}

\newtheorem{proposition}[theorem]{Proposition}

\newtheorem{lemma}[theorem]{Lemma}

\newtheorem{definition}[theorem]{Definition}
\newtheorem{definition/lemma}[theorem]{Definition/Lemma}

\theoremstyle{definition}
\newtheorem{example}{Example}

\title{Reduction rules for Demazure modules}
\author{Marc Besson}
\address{YMSC, Tsinghua University, Beijing 100871, P.R. China}
\email{bessonm@tsinghua.edu.cn}
\author{Sam Jeralds}
\address{University of Sydney, Camperdown, NSW 2006, Australia}
\email{samuel.jeralds@sydney.edu.au}
\author{Joshua Kiers}
\address{Epic Systems Corporation, Verona, WI 53593, USA}
\email{jkiers@epic.com}

\begin{document}
\begin{abstract}
For $G$ a complex reductive group and $B  \subseteq G$ a Borel subgroup, we provide a reduction rule for certain weight multiplicities in Demazure modules $V_\lambda^w$: given a weight $\mu$ on a face of the associated weight polytope $P_\lambda^w$, we reduce the computation of the dimension of the weight space $V_\lambda^w(\mu)$ to a similar problem of computing the weight space dimension for a Demazure module of a Levi subgroup of $G$. 
\end{abstract}
\maketitle

\section{Introduction}


\subsection{Reduction Rules}
For $G$ a complex reductive algebraic group and $G'\subseteq G$ a complex reductive subgroup, \textit{branching problems} seek to understand how representations of $G$ restrict to representations of $G'$. Classical examples include the study of characters of representations (for the case $G'=T$ a maximal torus) and the decomposition of tensor products of irreducible, highest weight representations (for the case $G=G' \times G'$). Various formulations of branching problems have been of central interest in representation theory, and a multitude of algebraic, combinatorial, and geometric tools have been developed to understand how these restrictions behave. 

In this note, we focus on a particular subclass of branching problems, known as \textit{reduction rules}. Reduction rules are equalities in branching problems--reducing a computation of some quantity (a dimension, a multiplicity, etc.) for $G$-representations to a related quantity in a $G'$-representation. To develop reduction rules, two key insights are typically required: first, how can one recognize (based on data from $G$) which specific quantities should satisfy a reduction rule, and second, to what quantities (based on $G'$) they should correspond. 

Many examples of reduction rules appear in the literature, with key motivating examples being rules for Kostka numbers and weight multiplicities \cite{BZ}, Littlewood--Richardson coefficients \cite{DW, KTT1}, tensor product multiplicities \cite{Roth} (and more general branching multiplicities \cite{Ress2}), and dimensions of conformal blocks \cite{Schu}. While proofs of each of these results vary, the reduction rules themselves share many commonalities: first, each reduction is to a similar quantity for a \textit{predictable} smaller rank group (often a Levi subgroup of $G$), and second, while not always phrased in this language, each quantity amenable to the reduction rule corresponds to some \textit{extremal data} for the group $G$; often these are related to faces of an ample cone.

\subsection{Ample cones and reduction rules}

Let $X$ be a projective variety with an action of $G$, and $\mathcal{L}$ a $G$-linearized line bundle on $X$. To this data, one can associate the open locus of semi-stable points of $X$ with respect to $\mathcal{L}$, $X^{ss}(\mathcal{L})$, defined by 
$$
X^{ss}(\mathcal{L}):=\{x \in X | \exists N \geq 1 \ \exists \sigma \in H^0(X, \mathcal{L}^{\otimes N}): \sigma(x) \neq 0\}.
$$
Via the Hilbert--Mumford criterion, existence of semi-stable points can be characterized by a set of inequalities coming from the $G$-linearization on $\mathcal{L}$ and cocharacters of $G$. This reinterprets detecting semi-stable points into a problem of convex geometry; by varying the line bundle $\mathcal{L} \in \mathrm{Pic}^G(X)^+_\mathbb{Q}$ across rational ample line bundles, we arrive at the $G$-ample cone 
$$
\mathcal{AC}^G(X):= \{ \mathcal{L} \in \mathrm{Pic}^G(X)^+_\mathbb{Q} | X^{ss}(\mathcal{L}) \neq \varnothing \}, 
$$
which is a pointed rational convex cone. 

For the above reduction rules, where $X$ is a flag variety (or related variety), the faces of the ample cones $\mathcal{AC}^G(X)$ correspond to Levi subgroups of $G$, and the reductions correspond to those $\mathcal{L}$ lying on faces of the ample cone. Translation back into representation theory is then accomplished by classical results like the Borel--Weil theorem and its generalizations.

\subsection{Reduction rules for Demazure modules}
Taking inspiration from these examples, we now introduce the focus of this note: namely, we seek a reduction rule for weight multiplicities in Demazure modules $V^w_\lambda$, recalled in Section \ref{Notation}. The ultimate goal is to provide a condition on pairs of weights $(\lambda, \mu)$ satisfying $\mathrm{dim} V^w_\lambda(\mu) \neq 0$ that yields an equality of the form $\mathrm{dim}V^w_\lambda(\mu) = \mathrm{dim}V^{w_L}_{\lambda_L}(\mu)$ for a Demazure module $V^{w_L}_{\lambda_L}$ for a Borel subgroup $B_L \subseteq B$ corresponding to a Levi subgroup $L \subseteq G$. 

For the first task, we utilize our previous work on weight polytopes of Demazure modules and their faces $\mathcal{F}(v, \eta)$ depending on a dominant coweight $\eta$ and minimum-length representative $v \in W^\eta$ (see Section \ref{Polytopes}). In \cite{BJK2}, \cite{BJK3}, we used GIT techniques to determine inequalities governing the $T$-ample cones of Schubert varieties $X_w$; by fixing a dominant highest weight $\lambda$, one obtains a compact convex polytope $P^w_\lambda$, the \textit{Demazure weight polytope}. Consistent with the current narrative, the faces of these polytopes are themselves Demazure weight polytopes for a Levi subgroup. The location of the weight $\mu$ in $P^w_\lambda$ relative to the faces provides us with our desired condition for reduction.

For the second task, we derive a reduction rule for weights $\mu$ lying on a prescribed face $\mathcal{F}(v,\eta)$ of a fixed Demazure weight polytope $P^w_\lambda$ in two stages. We first use algebro-geometric techniques, adapting Roth's proof of the reduction rule for Littlewood--Richardson coefficients \cite{Roth}, to get an equality of weight multiplicities for the Levi Demazure module and an ``intermediate" Demazure module $V^q_\lambda$. We then use convex-geometric arguments, along with classical properties of the Demazure operators and the character formula, to show an equality of weight multiplicities for the intermediate Demazure module $V^q_\lambda$ and $V^w_\lambda$. Taken together, these two results give the following desired reduction rule (cf. Theorem \ref{PreciseTheorem} in the text):

\begin{theorem} \label{MainTheoremIntro}
Let $\mu$ be a weight of the Demazure module $V^w_\lambda$ with $\mu$ lying on the face $\mathcal{F}(v,\eta)$ of the associated polytope $P^w_\lambda$. Then 

$$\mathrm{dim} V^w_\lambda(\mu) = \mathrm{dim} V^{w_L}_{\lambda_L}(\mu),$$
where $L:= \dot{v}L_\eta \dot{v}^{-1}$ is the conjugate of the Levi $L_\eta$ determined by $\eta$, $w_L:= v \pi_\eta(w^{-1} \ast v)^{-1} v^{-1}$, $\lambda_L:=v \pi^{\eta}(w^{-1} \ast v) \lambda$, and $V^{w_L}_{\lambda_L}$ the Demazure module for the Borel subgroup $B_L \subseteq L$. 
\end{theorem}

\subsection{Outline of the paper}
In Section 2, we fix our basic notations and conventions for algebraic groups, flag varieties and Schubert varieties, and representations. We give only a brief sketch, and direct readers to \cite{KumarBook} for a more comprehensive treatment and references. In Section 3, we recall the weight polytopes for Demazure modules and their inequalities, following \cite{BJK2, BJK3}, and give a precise statement of the main result of this note, Theorem \ref{PreciseTheorem}. Sections 4 and 5 taken together construct the proof of this result.

\subsection{Acknowledgements}

We thank Rekha Biswal, whose interests regarding \cite{BJK4} and possible extensions to Demazure modules led to the main question addressed in this paper.


\section{Notations and preliminaries on Demazure modules} \label{Notation}

\subsection{Notation for reductive groups}
We will let $G$ denote a connected reductive algebraic group over $\mathbb{C}$, and we will choose a maximal torus $T$ and a Borel subgroup $B$ so that we have $T \subset B \subset G$. We have weight and coweight lattices $X^*(T)$ and $X_*(T)$, with associated pairing $\langle ,  \rangle : X^*(T) \times X_*(T) \rightarrow \mathbb{Z}$. We write $\Phi$ and $\Phi^{\vee}$ for the root and coroot systems associated to $T \subset G$, so $(\Phi, X^*(T), \Phi^{\vee}, X_*(T))$ is a root datum for $G$. The choice of Borel subgroup determines positive roots $\Phi^+$ as well as a base $\Delta \subset \Phi^+$; we write $\Delta=\{ \alpha_1, \dots \alpha_n \}$ where the $\alpha_i$ are simple roots determined by $B$. We also have $(\Phi^{\vee})^+$ determined by $B$, with base $(\Delta^{\vee})=\{\alpha_1^{\vee}, \dots \alpha_n^{\vee} \}$, and the coefficients $\langle \alpha_i, \alpha_j^{\vee} \rangle $ yield the Cartan matrix associated to the semisimple group $(G,G)$.

In this paper we will often work with polytopes, so it will be useful to work over $\mathbb{Q}$ or $\mathbb{R}$, rather than work integrally. Thus we consider the rational weight and coweight lattices $X^*(T)_{\mathbb{Q}}= X^*(T) \otimes_{\mathbb{Z}} \mathbb{Q}$ and $X_*(T)_\mathbb{Q}=X_*(T) \otimes_{\mathbb{Z}} \mathbb{Q}$. The pairing $\langle, \rangle$ extends by linearity to these $\mathbb{Q}$ vector spaces, so we have $\langle ,\rangle_{\mathbb{Q}} :X^*(T)_{\mathbb{Q}} \times X_*(T)_{\mathbb{Q}} \rightarrow \mathbb{Q}$. We will make use of fundamental weights $\omega_i \in X^*(T)_{\mathbb{Q}}^+$ as well as fundamental coweights $x_i \in X_*(T)_{\mathbb{Q}}^+$; these are determined by \[ \langle \omega_i, \alpha_j^{\vee} \rangle_{\mathbb{Q}}= \delta_{i,j} \] and \[\langle \alpha_i, x_j \rangle_{\mathbb{Q}} = \delta_{i,j}\] respectively.

The dynamical method associates to any dominant coweight $\eta$ a parabolic subgroup \[P_{\eta}:= \{g \in G | \lim_{t \rightarrow 0} \eta(t)g\eta(t)^{-1} ~\mathrm{exists} \}\] and a Levi subgroup \[L_{\eta}:= \{g \in G |\forall t,  \eta(t) g \eta(t)^{-1}=g \}.\]
The roots belonging to $P_{\eta}$ and the root system of $L_{\eta}$ are easy to describe; we have \[\Phi(P_{\eta}, T)=\{ \alpha \in \Phi_G| \langle \alpha, \eta \rangle \geq 0 \}\] and \[\Phi_{\eta}:=\Phi(L_{\eta}, T)=\{\alpha \in \Phi_G| \langle \alpha, \eta \rangle = 0\}.\]

\subsection{Weyl group combinatorics}
Associated to $G$ we have the Weyl group \[W=W_G=N_G(T)/T\]  with Bruhat order $\leq$. Given $w\in W$, we write $\dot w$ for an element of $N_G(T)$ such so that $w=\dot w T$.  Our choice of B, and thus of simple roots $\alpha_i \in \Delta$, yields a system of simple reflections $S=\{s_i \} \subset W$ which generate $W$ under the usual Coxeter relations. We have as usual a length function $\ell: W \rightarrow \mathbb{Z}^{\geq 0}$ such that $\ell(s_i)=1$. The Weyl group $W$ acts on the root system, and we write \[R(w):=\{ \alpha \in \Phi^+ | w.\alpha \in -\Phi^+\};\] this is the inversion set of $w$.

For dominant coweights $\eta$,  we write $W_{\eta}$ for the parabolic subgroup of $W$ generated by the reflections $\{s_i\}$ such that $\alpha_i \in \Phi_{\eta}$. The length function $\ell$ and Bruhat order descend to a length function and Bruhat order on $W_{\eta}$. We write \[W^{\eta} := W/W_{\eta}\] 
for the set of cosets. Each coset $v W_{\eta}$ has a distinguished representative, called \textit{minimum-length representatives}. These can be characterized in the following way: $v \in W$ is a minimum length representative for $vW_{\eta}$ if \[v(\Phi_{\eta}^+) \subset \Phi_G^+.\] Given $w \in W$ and a parabolic Weyl subgroup $W_{\eta}$ we write $\pi^{\eta}(w)$ for a minimum length representative of $wW_{\eta}$. There exists a unique element $u$ of $W_{\eta}$ such that $w= \pi^{\eta}(w) \cdot u$ and $\ell(\pi^{\eta}(w))+\ell(u)=\ell(w)$. We write $\pi_{\eta}(w):=u$. So given $W_{\eta}$ and $w \in W$ we can always write \[w=\pi^{\eta}(w) \cdot \pi_{\eta}(w)\] such that \[\ell(w)=\ell(\pi^{\eta}(w))+\ell(\pi_{\eta}(w)).\]

Finally, we recall the Demazure product. Given $v, w \in W$, the collection of elements $\{xq| x \leq v, q \leq w\}$ has an element which is maximal with respect to the Bruhat order. We denote this element by $v * w$. This element may equivalently be characterized as the maximum element in the left translated interval $v [e, w]$ or as the maximum element in the right translated interval $[e,v] w$. For simple reflections $s_i$, these equivalent definitions yield 
$$
s_i \ast w = \begin{cases}
s_i w, & \ell(s_iw) =\ell(w)+1 \\
w, & \ell(s_i w) = \ell(w)-1 
\end{cases}
$$
and the equivalent inductive construction of the Demazure product, as follows: if $v=v's_i$ with $v' < v$, then 
$$
v \ast w = v' \ast (s_i \ast w).
$$

\subsection{Flag variety geometry}
We let $X$ denote the flag variety; due to our choice of $B$ we can write $X=G/B$. From now on we assume $G$ is simply connected. We may identify $Pic(X) = X^*(T)$; all line bundles $\mathscr{L}$ on $X$ are isomorphic to some $\mathscr{L}_{\lambda}$, the sheaf of sections of $G \times_B \mathbb{C}_{-\lambda}$. Here $\mathbb{C}_{-\lambda}$ is the one-dimensional representation of $B$ where $T$ acts via $\lambda$ and $U$ acts trivially and $G \times_B \mathbb{C}_{-\lambda}$ is the associated bundle construction. Since we have assumed $G$ is simply connected, we have a unique $G$-linearization on each $\mathscr{L}_{\lambda}$; in other words $Pic^G(X) \simeq X^*(T)$. In this work we usually consider $T$-linearizations; we have the short exact sequence \[0 \rightarrow X^*(T) \rightarrow Pic^{T}(X) \rightarrow Pic(X) \rightarrow 0.\] In other words a $T$-linearization on $\mathscr{L}_{\lambda}$ may be specified by an additional character $\mu \in X^*(T)$. An explicit description of the $\mu$-linearization on $\mathscr{L}_{\lambda}$ is given by \[t.[x, z]=[tx, \mu(t)z]\] for $[x,z]$ a point in $G \times_B \mathbb{C}_{-\lambda}$. 
Via the Borel-Weil theorem, and using the $G$-linearization on $\mathscr{L}_{\lambda}$, we have \[H^0(X, \mathscr{L}_{\lambda})^* \simeq V_{\lambda}\] as $G$-representations, where $V_{\lambda}$ denotes the irreducible algebraic representation of $G$ with highest weight $\lambda$. We write $V_{\lambda}(\mu)$ for the $\mu$-weight space of $V_{\lambda}$. Twisting the $T$-linearization on $\mathscr{L}_{\lambda}$ has the effect of shifting the $T$-weights on $V_{\lambda}$; if $\mathbb{L}=\mathscr{L}_{\lambda} \otimes \mathbb{C}_{\mu}$ we have \[(H^0(X, \mathbb{L})^*)^T \simeq V_{\lambda}(\mu).\]  By twisting the $T$-linearizations we may study any individual weight space $V_{\lambda}(\mu)$ as a space of $T$-invariants.

\subsection{Schubert Geometry and Demazure modules}
The Bruhat decomposition \[G=\bigcup_{w \in W} B \dot w B\] yields a stratification of the flag variety  \[X=\bigcup_{w \in W} B \dot wB/B\] by $B$ orbits which are also affine spaces: $B \dot{w} B/B \simeq \mathbb{A}^{l(w)}$. We write $\mathring{X}_w$ for this $B$ orbit, called a Schubert cell. The closure relations among the Schubert cells are given precisely by the Bruhat order and we write $X_w$ for the closure of $\mathring{X}_w$, these are the Schubert varieties. While in general not smooth, the $X_w$ are known to be normal, Cohen-Macaulay, and have rational singularities. We write \[i_w: X_w \rightarrow X\] for the closed embedding.

Demazure modules are $B$ modules which can be constructed as spaces of sections of restrictions of line bundles. Let $\lambda \in X^*(T)^+$, we have the Demazure module associated to $(\lambda, w)$: \[V_{\lambda}^w:=H^0(X_w, i_w^*\mathscr{L}_{\lambda})^*.\]

The characters of these modules can be straightforwardly computed via the Demazure character formula. We have the representation ring $R(T)=\mathbb{Z}[X^*(T)]$ and to each simple reflection $s_i$ we have a linear endomorphism \[D_{i} (x) = \frac{x-e^{-\alpha_i} s_i(x)}{1-e^{-\alpha_i}}\] for $x \in R(T)$. The Demazure character formula then asserts that for any reduced word $\mathfrak{w}=s_{i_1} \dots s_{i_k}$ for $w$ and $\lambda \in X^*(T)^+$ we have \[\mathrm{ch}(V_{\lambda}^w) = D_{i_1} \circ D_{i_2} \circ \dots \circ D_{i_k} (e^{\lambda}).\]

In fact, $s_i \mapsto D_i$ is a representation of the Demazure monoid $(W, *)$, with the Demazure product. In particular, if $\ell(s_i w) =\ell(w)-1$, then $s_i *w=w$ and \[D_{i} \left(\mathrm{ch}(V_{\lambda}^w)\right) = \mathrm{ch}(V_{\lambda}^w).\]
There is also a universal enveloping algebra description of $V_{\lambda}^w$. If $v_{w \lambda}$ is a nonzero extremal weight vector in $V_{\lambda}(w \lambda)$, then we can describe $V_{\lambda}^w$ as the $U(\mathfrak{b})$-submodule generated by $v_{w \lambda}$, where $\mathfrak{b}$ is the Lie algebra of $B$:  \[V_{\lambda}^w = U(\mathfrak{b}) . v_{w \lambda} \subset V_{\lambda}.\]The properties of the Demazure product above can be expressed in the following way. Let $\mathfrak{sl}_{2, \alpha_i}$ denote the embedding of $\mathfrak{sl}_{2}$ into $\mathfrak{g}$ such that $e_i \mapsto e_{\alpha_i}$, $h_{i } \mapsto h_{\alpha_i}$ etc. Then for a reduced word $s_{i_1} \dots s_{i_k}$ for $w$, such that $s_i * w=w$, \[\mathrm{ch}(V_{\lambda}^w) \in R(T)\] is in the image of \[Rep^{fd}(\mathfrak{sl}_{2, \alpha_i}) \rightarrow R^{fd}(\mathfrak{sl}_{2, \alpha_i}) \rightarrow R(T).\] Said a little less pedantically, $\mathrm{ch}(V_{\lambda}^w)$ is an honest (not virtual) character of $\mathfrak{sl}_{2, \alpha_i}$.


\section{Weight polytopes for Demazure modules} \label{Polytopes}

We recall in this section, keeping notations as introduced, the construction and properties of the \textit{Demazure weight polytope} $P_\lambda^w$ associated to the Demazure module $V_\lambda^w$. For a general $T$-module $M$, with $T$ a torus acting semi-simply, understanding the set of weights 
$$
\mathrm{wt}(M):=\{\mu \in X^\ast(T): M(\mu) \neq 0 \}
$$
is a foundational question in representation theory and character theory. As this set is some collection of points in the lattice $X^\ast(T)$, a rough approximation of this data can be obtained by taking the associated weight polytope of $M$, which is the convex hull 
$
\mathrm{conv}_{\mathbb{Q}}\left(\mathrm{wt}(M)\right),
$
which (for the purposes of this paper) we recognize as a compact, convex polytope in $X^\ast(T)_\mathbb{Q}$. Applied to Demazure modules, we get the following objects of central interest:

\begin{definition} Let $V_\lambda^w$ be a Demazure module. The Demazure weight polytope $P_\lambda^w$ is given by $P_\lambda^w:=\mathrm{conv}_{\mathbb{Q}}\left(\mathrm{wt}(V_\lambda^w)\right)$.
\end{definition}

The original interest for Demazure weight polytopes was how close of a ``rough approximation" these were to the character data of the Demazure module $V_\lambda^w$; this question was posed in the language of \textit{saturation of Demazure characters} in \cite{BJK2}, with the conjectural equivalence 
$$
V_\lambda^w (\mu) \neq 0 \iff \mu \in P_\lambda^w \cap (\lambda + Q).
$$
For irreducible, highest-weight $G$-modules (corresponding to the $w=w_0$ case), this equivalence is a classical textbook result. For general Demazure modules, we proved that this equivalence holds for $G$ of classical type and exceptional types $F_4$ or $G_2$ \cite{BJK2}*{Theorem 1.1}, and the third author has checked that this equivalence holds for type $E_6$. Types $E_7$ and $E_8$, along with a more uniform proof, remain open.

\subsection{Faces and Inequalities of $P_\lambda^w$}

We recall now from \cite{BJK2} and \cite{BJK3} the relevant combinatorial and convex-geometric features of $P_\lambda^w$ which will be of use for the reduction rule of Theorem \ref{PreciseTheorem}.

As compact convex polytopes can be defined both by their vertices as well as their inequalities, a first pass at understanding the combinatorial structure of $P_\lambda^w$ is given by the following lemma, which appears in \cite{BJK2}*{Theorem 4.7} but is also implicit in the work of Dabrowski \cite{Dab}*{page 119}. 

\begin{lemma} \label{Vertices}
For any $\lambda \in X^\ast(T)^+$ and $w \in W$, we have 
$$
P_\lambda^w = \mathrm{conv}_\mathbb{Q} \left(\left \{x \lambda | x \leq w \right \} \right).
$$
\end{lemma}

Lemma \ref{Vertices} gives that the vertices of the polytope $P_\lambda^w$ are determined by a subset of the \textit{extremal weights} $x\lambda \in X^\ast(T)$. This description simultaneously connects Demazure weight polytopes with objects of interest in the literature including pseudo-Weyl polytopes, generalized Coxeter permutohedra, and Bruhat interval polytopes. The relationship to the latter will be of note, as the structure of the faces of Demazure weight polytopes shares key properties with the more familiar Bruhat interval polytopes of the symmetric group. 

One immediate consequence of Lemma \ref{Vertices} is the following corollary.

\begin{corollary} \label{StablePolytope}

Let $s \in W$ be a simple reflection with $sw < w$. Then $P_\lambda^w$ is stable under the action of $s$: $s(P_\lambda^w)=P_\lambda^w$.

\end{corollary}

\begin{proof}
Note that by the so-called ``Z-property" or ``diamond lemma" for the Bruhat order \cite{BGG}*{Lemma 2.5}, if $sw < w$ then $s\left( \{x: x \leq w\} \right) = \{x : x \leq w\}$, so the lower Bruhat interval $[e,w]$ is stable under left-multiplication by $s$. Then clearly 
$$
s(P_\lambda^w) = \mathrm{conv}_\mathbb{Q}\left(\{sx\lambda : x \leq w\} \right) = P_\lambda^w.
$$
\end{proof}
\noindent Alternatively, this follows from the fact that if $s=s_i$ is the simple reflection associated to simple root $\alpha_i$, $V_\lambda^w$ is an $\mathfrak{sl}_{2, \alpha_i}$-module.

On the other hand, the description of $P^w_\lambda$ in terms of its inequalities was given in \cite{BJK2} using methods of Geometric Invariant Theory (GIT), via the Hilbert--Mumford criterion for semistability. Briefly, given any dominant cocharacter $\eta \in X_\ast(T)^+$ and $v \in W^\eta$ a minimum-length coset representative, the Hilbert--Mumford criterion gives a numerical condition for checking the semistability of a point $x \in X_w$ in the Schubert variety. This in turn yields a set of inequalities for the possible weights of $V^w_\lambda$:
$$
\mu \in P^w_\lambda \text{ if and only if }
\langle \lambda, (w^{-1} \ast v) \eta \rangle \leq \langle \mu, v \eta \rangle
$$
for all choices of $\eta$ and $v$ as above. In a natural way, we can reduce this list of inequalities to those which parametrize the \textit{facets}, by reducing to the case of fundamental coweights \cite{BJK2}*{Theorem 6.9}.

\begin{theorem} \label{FacetInequalities}
Let $\lambda$ be a dominant weight and $w \in W$. Then $\mu \in P^w_\lambda$ if and only if, for every maximal parabolic $P=P_i$ associated to the fundamental coweight $x_i$ and $v \in W^P$, the inequality 
$$
\langle \lambda, (w^{-1} \ast v)x_i \rangle \leq \langle \mu, vx_i \rangle
$$
holds. 
\end{theorem}

Note that, in the statement of Theorem \ref{FacetInequalities} and in the preceding discussion, we could without change have reduced to the minimum-length representative $\pi^\eta(w^{-1} \ast v) \in W^\eta$ and $\pi^P(w^{-1} \ast v) \in W^P$ in place of $(w^{-1} \ast v)$ in the inequalities. This change, while not affecting the inequalities, is useful in understanding the structure of the faces of $P^w_\lambda$. To this end, we label arbitrary faces (not just facets) by $v$ and $\eta$ determining these inequalities.

\begin{definition} \label{FaceDef}
Let $\eta \in X_\ast(T)^+$ be a dominant coweight, $v \in W^\eta$ a minimum-length coset representative. We denote by $\mathcal{F}(v,\eta)$ the face of $P^w_\lambda$ given by 
$$
\mathcal{F}(v,\eta):=\{\mu \in P^w_\lambda | \langle \lambda, \pi^\eta(w^{-1} \ast v) \eta \rangle = \langle \mu, v \eta \rangle \}.
$$
\end{definition}

The faces $\mathcal{F}(v,\eta)$ have a rich combinatorial structure, which can be used to identify them with convex hulls of orbits for lower intervals in parabolic subgroups $W_\eta \subseteq W$ \cite{BJK2}*{Proposition 7.4}, \cite{BJK3}*{Theorem 10.8}; this also, as is expected by the similarity to the Bruhat interval polytopes of Tsukerman and Williams \cite{TW}, demonstrates the philosophy that faces of Bruhat interval polytopes are themselves Bruhat interval polytopes. 

%
%

At the representation-theoretic level, the combinatorial identification can be upgraded to interpret these faces $\mathcal{F}(v,\eta)$ as Demazure weight polytopes themselves. Let $L_\eta \subseteq G$ be the standard Levi determined by the dominant coweight $\eta$ and $B_\eta:=B \cap L_\eta \subseteq B$ its Borel subgroup. We set $L:=\dot{v}L_\eta \dot{v}^{-1}$ to be the conjugate of the Levi by $v \in W^\eta$, and note that its Borel subgroup $B_L = \dot{v} B_\eta \dot{v}^{-1} \subseteq B$ as $v \in W^\eta$. Then $V^w_\lambda$ is canonically a $B_L$-module. The following proposition crucially lets us compare a Demazure module for $B_L$ and the Demazure module $V^w_\lambda$ for $B$ which recovers the relationship between the face $\mathcal{F}(v,\eta)$ and the polytope $P^w_\lambda$. We sketch the proof given in \cite{BJK2} for the convenience of the reader.

\begin{proposition}{\cite{BJK2}*{Proposition 7.9}} \label{LeviModuleInclusion}

Let $w_L:=v \left(\pi_\eta(w^{-1} \ast v ) \right)^{-1} v^{-1} $, $\lambda_L:= v \left(\pi^\eta (w^{-1} \ast v) \right)^{-1} \lambda $, and let $q \in [e,w]$ be the unique element of the Weyl group such that $w^{-1} \ast v= q^{-1} v$. Then every map in the following commutative diagram is a $B_L$-equivariant inclusion: 
\begin{center}
\begin{tikzcd}
V_{\lambda_L}^{w_L} \arrow[r] \arrow[d]& V_\lambda^q \arrow[r] & V_{\lambda}^w \arrow[d] \\
V_{\lambda_L} \arrow[rr] & &  V_\lambda
\end{tikzcd}
\end{center}
where, in particular, we view $V^{w_L}_{\lambda_L}$ and $V_{\lambda_L}$ as Demazure and irreducible highest weight modules for $B_L$ and $L$, respectively. 
\end{proposition}

\begin{proof}
It is readily checked that $\lambda_L$ is a highest weight for the (possibly reducible) $L$-module $V_\lambda$, as $B_L$ stabilizes the line spanned by $v\left(\pi^\eta(w^{-1}\ast v) \right)^{-1} v_\lambda$. Thus there is an irreducible $L$-submodule of $V_\lambda$ isomorphic to $V_{\lambda_L}$; this must be the unique such submodule, as $$\dim V_\lambda(v\left(\pi^\eta(w^{-1}\ast v) \right)^{-1} \lambda) = 1$$ and $L$ contains the full torus. 

Then sitting inside of $V_{\lambda_L} \subseteq V_\lambda$ we form the $B_L$ Demazure module $V_{\lambda_L}^{w_L}$; note that 
$$
\begin{aligned}
w_L \lambda_L &= v \left(\pi_\eta(w^{-1} \ast v) \right)^{-1} v^{-1} v \left( \pi^\eta( w^{-1} \ast v) \right)^{-1} \lambda \\
&= v(w^{-1} \ast v)^{-1}\lambda \\
&= q\lambda
\end{aligned}
$$
by choice of $q \leq w$ with $q^{-1}v=w^{-1} \ast v$. Then $V^{w_L}_{\lambda_L}$ is the smallest $B_L$-submodule containing $q\lambda$, hence the smallest $B_L$-submodule of $V_\lambda$ containing $q\lambda$, so naturally sits inside the $B$-submodule $V^q_\lambda$. Finally, as $q \leq w$, $V^q_\lambda \subseteq V^w_\lambda$. 
\end{proof}

As a key takeaway from the previous proof, the lowest weight of the $B_L$ Demazure module $V^{w_L}_{\lambda_L}$ is precisely $q\lambda$, and is the lowest weight of $P^w_\lambda$ on the face $\mathcal{F}(v,\eta)$. An immediate corollary to this result is the following identification. 

\begin{corollary}{\cite{BJK2}*{Corollary 7.7}}
Let $P_{\lambda_L}^{w_L}$ be the Demazure weight polytope for the $B_L$-module $V^{w_L}_{\lambda_L}$. Then we have an equality of convex polytopes
$$
P_{\lambda_L}^{w_L} = \mathcal{F}(v, \eta).
$$
\end{corollary}

With all of this in hand, we can now frame our desired reduction rule for Demazure weight multiplicities. Fixing $\lambda \in X^\ast(T)^+$ and $w \in W$, let $\mu \in X^\ast(T)$ be such that $V^w_\lambda(\mu) \neq 0$, so that $\mu \in P^w_\lambda$. Suppose further that $\mu \in \mathcal{F}(v,\eta)$ for some dominant coweight $\eta \in X_\ast(T)^+$ and $v \in W^\eta$; in particular, $\mu \in P^{w_L}_{\lambda_L}$. Our primary result compares the weight multiplicity of $\mu$ in $V^w_\lambda$ and $V^{w_L}_{\lambda_L}$ in two stages, as follows.

%
%
%

\begin{theorem} \label{PreciseTheorem}

Retain all notation as in Proposition \ref{LeviModuleInclusion}, and let $\mu$ be a weight of $V^w_\lambda$ with $\mu \in \mathcal{F}(v, \eta)$. Then 

\begin{itemize}

\item[(1)]  $\dim V_{\lambda_L}^{w_L}(\mu) = \dim V_\lambda^q(\mu)$, where we view $V_{\lambda_L}^{w_L}$ as a Demazure module for $B_L$ and $V^q_\lambda$ as a Demazure submodule for $B$. 

\item[(2)] $\dim V_\lambda^q(\mu) = \dim V_\lambda^w(\mu)$, viewing both as Demazure modules for $B$. 

\end{itemize}

\end{theorem}

We construct the proof of these two statements subsequently in the next two sections.



\begin{example}
Before proceeding with the proof, we fix an instructive example of Theorem \ref{PreciseTheorem}. Let $G$ be the simply-connected group of type $B_3$ ($=\op{Spin}(7)$) and $\lambda=\rho=\omega_1+\omega_2+\omega_3$. We set $w=s_1s_3s_2s_3s_1s_2s_3$ in the Weyl group, and consider the face $\mathcal{F}(s_1, x_1)$ of $P_\lambda^w$ where $x_1$ is the first fundamental coweight. Then one can check that $q=s_3s_2s_3s_1s_2s_3$. Figure \ref{example-figure} shows the polytopes $P_\lambda^q$ and $P_\lambda^w$ with their common face $\mathcal{F}(s_1,x_1)$ highlighted. Figure \ref{levi-figure} isolates this face as the Demazure weight polytope $P_{\lambda_L}^{w_L}$ for a Levi $L$ whose semisimple component is of type $B_2$. With respect to this Levi, the highest weight $\lambda_L$ and Weyl group element $w_L$ correspond to $2\omega^L_1+\omega^L_2$ and $t_2t_1t_2$ respectively, where $t_1=s_1s_2s_1$ and $t_2=s_3$ are the simple reflections in the Weyl group of $L$.  
\end{example}


\tdplotsetmaincoords{70}{35}
\begin{figure}
\begin{center}
\begin{tikzpicture}[tdplot_main_coords]

\coordinate(s3s2s3s1s2s3s2) at (-3/2, -1/2, -5/2);
\coordinate(s2s3s1s2s3s2) at (-3/2, -1/2, 5/2);
\coordinate(s3s1s2s3s1s2) at (-3/2, 1/2, -5/2);
\coordinate(s3s2s3s1s2s3) at (-1/2, -3/2, -5/2); 
\coordinate(s3s1s2s3s2) at (-3/2, 5/2, -1/2);
\coordinate(s1s2s3s1s2) at (-3/2, 1/2, 5/2);
\coordinate(s2s3s1s2s3) at (-1/2, -3/2, 5/2); 
\coordinate(s3s2s3s1s2) at (1/2, -3/2, -5/2); 
\coordinate(s3s1s2s3s1) at (-1/2, 3/2, -5/2);
\coordinate(s1s2s3s2) at (-3/2, 5/2, 1/2);
\coordinate(s3s2s3s2) at (5/2, -3/2, -1/2); 
\coordinate(s3s1s2s3) at (-1/2, 5/2, -3/2);
\coordinate(s2s3s1s2) at (1/2, -3/2, 5/2); 
\coordinate(s1s2s3s1) at (-1/2, 3/2, 5/2);
\coordinate(s3s1s2s1) at (1/2, 3/2, -5/2);
\coordinate(s3s2s3s1) at (3/2, -1/2, -5/2);
\coordinate(s2s3s2) at (5/2, -3/2, 1/2); 
\coordinate(s3s1s2) at (1/2, 5/2, -3/2);
\coordinate(s1s2s3) at (-1/2, 5/2, 3/2);
\coordinate(s3s2s3) at (5/2, -1/2, -3/2);
\coordinate(s1s2s1) at (1/2, 3/2, 5/2);
\coordinate(s2s3s1) at (3/2, -1/2, 5/2);
\coordinate(s3s2s1) at (3/2, 1/2, -5/2);
\coordinate(s3s2) at (5/2, 1/2, -3/2);
\coordinate(s2s3) at (5/2, -1/2, 3/2);
\coordinate(s1s2) at (1/2, 5/2, 3/2);
\coordinate(s3s1) at (3/2, 5/2, -1/2);
\coordinate(s2s1) at (3/2, 1/2, 5/2);
\coordinate(s2) at (5/2, 1/2, 3/2);
\coordinate(s3) at (5/2, 3/2, -1/2);
\coordinate(s1) at (3/2, 5/2, 1/2);
\coordinate(e) at (5/2, 3/2, 1/2);


\node at (s3s2s3s1s2s3s2) {\small $\bullet$};
\node at (s2s3s1s2s3s2) {\small $\bullet$};
\node at (s3s2s3s1s2s3) {\small $\bullet$};
\node at (s1s2s3s1s2) {\small $\bullet$};
\node at (s2s3s1s2s3) {\small $\bullet$};
\node at (s3s2s3s1s2) {\small $\bullet$};
\node at (s3s2s3s2) {\small $\bullet$};
\node at (s2s3s1s2) {\small $\bullet$};
\node at (s1s2s3s1) {\small $\bullet$};
\node at (s3s2s3s1) {\small $\bullet$};
\node at (s2s3s2) {\small $\bullet$};
\node at (s3s2s3) {\small $\bullet$};
\node at (s1s2s1) {\small $\bullet$};
\node at (s2s3s1) {\small $\bullet$};
\node at (s3s2s1) {\small $\bullet$};
\node at (s3s2) {\small $\bullet$};
\node at (s2s3) {\small $\bullet$};
\node at (s1s2) {\small $\bullet$};
\node at (s2s1) {\small $\bullet$};
\node at (s2) {\small $\bullet$};
\node at (s3) {\small $\bullet$};
\node at (s1) {\small $\bullet$};
\node at (e) {\small $\bullet$};


\draw[thick,fill=orange,opacity=0.5] (s3s2s3s1s2s3) -- (s3s2s3s1s2) -- (s3s2s3s2) -- (s2s3s2)
  -- (s2s3s1s2) -- (s2s3s1s2s3) -- (s3s2s3s1s2s3); 
  

\draw[thick,fill=lightgray,opacity=0.5] (s3s2s3s1) -- (s3s2s1) -- (s3s2) -- (s3s2s3) -- (s3s2s3s1);
\draw[thick,fill=lightgray,opacity=0.5] (s2s3s1s2) -- (s2s3s1) -- (s2s3) -- (s2s3s2) -- (s2s3s1s2);
\draw[thick,fill=lightgray,opacity=0.5] (s3s2s3s1s2) -- (s3s2s3s1) -- (s3s2s3) -- (s3s2s3s2) -- (s3s2s3s1s2);

\draw[thick,fill=lightgray,opacity=0.5]   (s1s2s3s1) -- (s1s2s1) -- (s2s1) -- (s2s3s1) -- (s2s3s1s2) -- (s2s3s1s2s3) ;
\draw[thick,dashed] (s2s3s1s2s3s2) -- (s1s2s3s1s2) -- (s1s2s3s1) -- (s1s2s1) -- (s2s1) -- (s2s3s1) -- (s2s3s1s2) -- (s2s3s1s2s3) -- (s2s3s1s2s3s2);

\draw[thick,dashed] (-3/2, 1/2, 5/2) -- (-3/2, 1/2, 1.85);
\draw[thick,dashed] (-3/2, -1/2, -5/2) -- (-3/2, -0.1, -5/2);

\draw[thick,fill=lightgray,opacity=0.5] (s2s3s1) -- (s2s1) -- (s2) -- (s2s3) -- (s2s3s1);

\draw[thick,dashed] (s3s2s3s1s2s3s2) -- (s2s3s1s2s3s2) -- (s2s3s1s2s3) -- (s3s2s3s1s2s3) -- (s3s2s3s1s2s3s2);

\draw[thick,fill=lightgray,opacity=0.5] (s1s2s1) -- (s1s2) -- (s1) -- (e) -- (s2) -- (s2s1) -- (s1s2s1);
\draw[thick,fill=lightgray,opacity=0.5] (s3s2s3s2) -- (s2s3s2) -- (s2s3) -- (s2) -- (e) -- (s3) -- (s3s2) -- (s3s2s3) -- (s3s2s3s2);

\node[right] at (e) {$\lambda$};

\node[below left] at (s3s2s3s1s2s3) {$q\lambda=w_L \lambda_L$};
\node[right] at (s2s3s2) {$\lambda_L$};


\coordinate(w1) at (-1/2, -3/2, 5/2);
\coordinate(w2) at (-1/2, -3/2, 3/2);
\coordinate(w3) at (-1/2, -3/2, 1/2);
\coordinate(w4) at (-1/2, -3/2, -1/2);
\coordinate(w5) at (-1/2, -3/2, -3/2);
\coordinate(w6) at (-1/2, -3/2, -5/2);
\coordinate(w7) at (1/2, -3/2, 5/2);  
\coordinate(w8) at (1/2, -3/2, 3/2);  
\coordinate(w9) at (1/2, -3/2, 1/2);  
\coordinate(w10) at (1/2, -3/2, -1/2);  
\coordinate(w11) at (1/2, -3/2, -3/2);  
\coordinate(w12) at (1/2, -3/2, -5/2);
\coordinate(w13) at (3/2, -3/2, 3/2);  
\coordinate(w14) at (3/2, -3/2, 1/2);  
\coordinate(w15) at (3/2, -3/2, -1/2);  
\coordinate(w16) at (3/2, -3/2, -3/2);
\coordinate(w17) at (5/2, -3/2, 1/2);
\coordinate(w18) at (5/2, -3/2, -1/2);

\node at (w1) {\small $\bullet$};
\node at (w2) {\small $\bullet$};
\node at (w3) {\small $\bullet$};
\node at (w4) {\small $\bullet$};
\node at (w5) {\small $\bullet$};
\node at (w6) {\small $\bullet$};
\node at (w7) {\small $\bullet$};
\node at (w12) {\small $\bullet$};
\node at (w13) {\small $\bullet$};
\node at (w16) {\small $\bullet$};
\node at (w17) {\small $\bullet$};
\node at (w18) {\small $\bullet$};

\filldraw (w8) circle[radius=2pt];
\draw (w8) circle[radius=4pt];

\filldraw (w9) circle[radius=2pt];
\draw (w9) circle[radius=4pt];

\filldraw (w10) circle[radius=2pt];
\draw (w10) circle[radius=4pt];

\filldraw (w11) circle[radius=2pt];
\draw (w11) circle[radius=4pt];

\filldraw (w14) circle[radius=2pt];
\draw (w14) circle[radius=4pt];

\filldraw (w15) circle[radius=2pt];
\draw (w15) circle[radius=4pt];

\end{tikzpicture}
\caption{\label{example-figure} Demazure polytopes for type $B_3$, highest weight $\lambda = \rho = \omega_1+\omega_2+\omega_3$. The shaded region is for $q = s_3s_2s_3s_1s_2s_3$; the wire frame extends to the polytope for $w =  s_1s_3s_2s_3s_1s_2s_3$. The face $\mathcal{F}(s_1,x_1)$ corresponding to $\eta = x_1, v = s_1$ is common to both polytopes and has been highlighted. The weight spaces on that face have been emphasized, where a single $\bullet$ stands for multiplicity one, and each concentric ring indicates an increase in multiplicity from there. }
\end{center}
\end{figure}
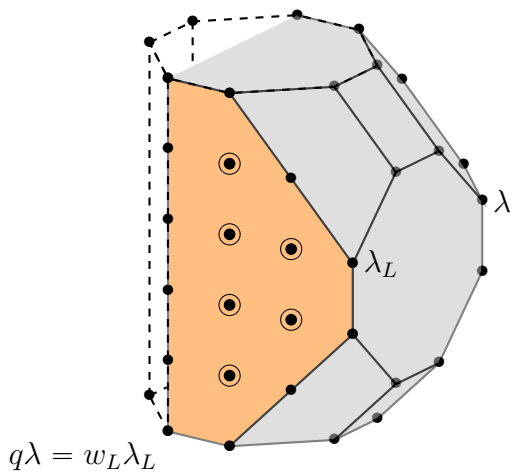


\begin{figure}
\begin{center}

\begin{tikzpicture}

\coordinate (w1) at (-1/2, 5/2);
\coordinate (w2) at (-1/2, 3/2);
\coordinate (w3) at (1/2, 5/2);
\coordinate (w4) at (-1/2, 1/2);
\coordinate (w5) at (1/2, 3/2);
\coordinate (w6) at (-1/2, -1/2);
\coordinate (w7) at (1/2, 1/2);
\coordinate (w8) at (3/2, 3/2);
\coordinate (w9) at (-1/2, -3/2);
\coordinate (w10) at (1/2, -1/2);
\coordinate (w11) at (3/2, 1/2);
\coordinate (w12) at (-1/2, -5/2);
\coordinate (w13) at (1/2, -3/2);
\coordinate (w14) at (3/2, -1/2);
\coordinate (w15) at (5/2, 1/2);
\coordinate (w16) at (1/2, -5/2);
\coordinate (w17) at (3/2, -3/2);
\coordinate (w18) at (5/2, -1/2);

\draw[thick,fill=orange,opacity=0.5] (w1) -- (w3) -- (w15) -- (w18) -- (w16) -- (w12) -- (w1) ; 

\filldraw (w1) circle[radius=2pt];
\filldraw (w2) circle[radius=2pt];
\filldraw (w3) circle[radius=2pt];
\filldraw (w4) circle[radius=2pt];

\filldraw (w5) circle[radius=2pt];
\draw (w5) circle[radius=4pt];

\filldraw (w6) circle[radius=2pt];

\filldraw (w7) circle[radius=2pt];
\draw (w7) circle[radius=4pt];

\filldraw (w8) circle[radius=2pt];
\filldraw (w9) circle[radius=2pt];

\filldraw (w10) circle[radius=2pt];
\draw (w10) circle[radius=4pt];

\filldraw (w11) circle[radius=2pt];
\draw (w11) circle[radius=4pt];

\filldraw (w12) circle[radius=2pt];
\node[below left] at (w12) {$s_2s_1s_2(2\omega_1+\omega_2)$};

\filldraw (w13) circle[radius=2pt];
\draw (w13) circle[radius=4pt];

\filldraw (w14) circle[radius=2pt];
\draw (w14) circle[radius=4pt];

\filldraw (w15) circle[radius=2pt];
\node[right] at (w15) {$2\omega_1+\omega_2$};

\filldraw (w16) circle[radius=2pt];
\filldraw (w17) circle[radius=2pt];
\filldraw (w18) circle[radius=2pt];

\end{tikzpicture}
\caption{\label{levi-figure} Demazure module of type $B_2$ with highest weight $2\omega_1 + \omega_2$ and Weyl group element $s_2s_1s_2$. }
\end{center}
\end{figure}
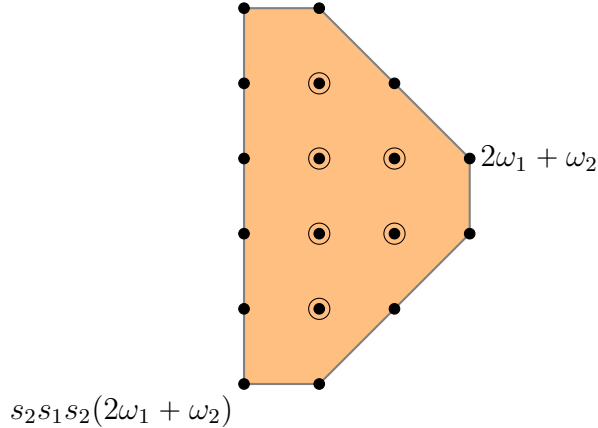

\section{Proof of Theorem \ref{PreciseTheorem}(1)}

The proof of Theorem \ref{PreciseTheorem}(1) given in this section is greatly influenced by Roth \cite{Roth},  as will be clear to any informed reader. We state, where possible, intermediate results in greater generality than necessary; this is in part to lessen the notational burden and will be connected precisely back to the setting at hand.

\subsection{Embeddings of Schubert varieties} Let $P$ be a (not necessarily standard) parabolic in $G$ with Levi $L$ such that our fixed maximal torus $T \subset G$ is a maximal torus of $L$. We then have an inclusion 

\[ W_L \hookrightarrow W\]
realizing $W_L$ as a reflection subgroup of $W$. We have the flag variety $L/B_L$, and for each element $y \in W_L$, we have the $L$-Schubert varieties $X^L_y$ and $L$-Schubert cells $\mathring{X}^L_y$. 

We first study a family of closed immersions 

\[i_{L,z}: L/B_L \rightarrow G/B\]

associated to certain elements $z \in W$.

\begin{lemma}
	Let $L/B_L$ be as above. Consider $z \in W$ such that $\dot{z}^{-1}B_L \dot{z} \subseteq B$ for any lift $\dot{z} \in L \subseteq G$. Then the map \[i_{L,z}: L/B_L \rightarrow G/B\] defined by \[lB_L\mapsto lzB\] is a closed immersion which maps $L-$Schubert cells $\mathring{X}^L_y$ to closed linear subschemes of $\mathring{X}_{yz}=\mathring{X}^G_{yz}$. In fact, $i_{L,z}$ induces an isomorphism $\mathring{X}^L_y \cong B_L yz B/ B =: \mathring{Y}^L_{yz}$. 
\end{lemma}

\begin{proof}

The algebraic map $L\to G$ given by $l \mapsto lz$ induces a map $L/B_L \to G/B$: indeed, if $l' = lb$ for $b \in B_L$, then $l'z = lbz = lz z^{-1} b z$, and $z^{-1}bz \in B$ by assumption. 

The restriction of this map to each Schubert cell $B_L y B_L/B_L$ induces closed immersions $B_L y B_L/B_L \to B_L y z B/B$ with disjoint images (injectivity follows from the fact that if $\gamma  \in R(y^{-1})$ then $\gamma \in R(z^{-1}y^{-1})$ by assumption on $z$).

Therefore $i_{L,z}$ induces an isomorphism of $B_L y B_L/B_L$ onto $B_L y z B/B$. 
\end{proof}

As an immediate corollary, we have the following:

\begin{corollary}
	Let $x$ be an arbitrary element of $W$ and let $L$, $B_L$ as above. Consider the Zariski closure \[Y^L_x:=\overline{B_L x B/B}\] of a $B_L$ orbit through $xB/B$. Then there exists an $L$-Schubert variety $X_{y}^L$ and a $z \in W$ such that \[Y^L_x= i_{L,z}(X^L_{y}) \simeq X^L_{y}.\]
\end{corollary}

\begin{proof}
Write $x^{-1}=\pi^L(x^{-1})\pi_L(x^{-1}) \in W^L \cdot W_L$, and set $y:=(\pi_L(x^{-1}))^{-1} \in W_L$ and $z:=(\pi^L(x^{-1}))^{-1}\in W$. Then indeed $z^{-1}$ satisfies $\dot{z}^{-1}B_L \dot{z} \subseteq B$, and $yz=x$ as desired. 
\end{proof} 

\subsection{Action on the Picard group}

The map $i_{L,z}: L/B_L \to G/B$ induces a map $i_{L,z}^*$ on the Picard groups. We wish to identify the image $i_{L,z}^*(\mathscr{L}_\lambda)$ starting with $\mathscr{L}_\lambda$ on $G/B$. 

Recall the fibre-bundle construction of the total space of $\mathscr{L}_\lambda$: 
$$G \times_B \C_{-\lambda} $$
where 
$(g,t) \sim (gb, \lambda(b) t)$ for all $b\in B$. 

\begin{proposition} \label{pullback}
We can identify pullbacks of line bundles as follows:
$i_{L,z}^*(\mathscr{L}_\lambda) = \mathscr{L}_{z\lambda}^L$. 
\end{proposition}

\begin{proof}
Define a map 
$$ L \times_{B_L} \C_{-z\lambda} \to G \times_B \C_{-\lambda}$$
given on coset representatives by 
$$
(l, t) \mapsto (lz, t).
$$

This is well-defined due to
$
(lb, (z\lambda)(b) t) \mapsto (lbz, (z\lambda)(b) t)$, the latter equal to $$ (lz z^{-1}bz, (z\lambda)(b) t) = (lz, \lambda(z^{-1} b z)^{-1} (z\lambda)(b) t ) = (lz, t).
$$

This identifies the total space of the pull-back $i_{L,z}^*(\mathscr{L}_\lambda)$ as $L \times_{B_L} \C_{-z\lambda} $. 
\end{proof}

Let $\mathscr{I}_{L,z}$ denote the ideal sheaf of $i_{L,z}(L/B_L)$ in $G/B$. Tensoring with $\mathscr{L}_{\lambda}$ we have the standard short exact sequence of sheaves on $G/B$:

\[0 \rightarrow \mathscr{I}_{L,z} \otimes \mathscr{L}_{\lambda} \rightarrow \mathscr{L}_{\lambda} \rightarrow (i_{L,z})_*\mathscr{O}_{L/B_L} \otimes \mathscr{L}_{\lambda} \rightarrow 0.\]

Upon taking cohomology we have the following 

\begin{proposition}
We have a short exact sequence 

\[0 \rightarrow H^0(G/B, \mathscr{I}_{L,z} \otimes \mathscr{L}_{\lambda}) \rightarrow H^0(G/B, \mathscr{L}_{\lambda}) \rightarrow H^0(L/B_L, i_{L,z}^*\mathscr{L}_{\lambda}) \rightarrow 0.\]
\end{proposition}

\begin{proof}
First, note that $z\lambda$ is a dominant weight for $L$ with respect to $B_L$, as $\lambda$ is dominant for $G$ with respect to $B$, $(z\lambda)(b)=\lambda(z^{-1}bz)$ for any $b \in B_L$, and $z^{-1}bz \in B$ by assumption on $z$. Thus $H^0(L/B_L, i^\ast_{L,z} \mathscr{L}_\lambda) \neq 0$. 

The only other point which needs justification is the surjection. We may regard the morphism $H^0(G/B, \mathscr{L}_{\lambda}) \rightarrow H^0(L/B_L, i^*_{L,z} \mathscr{L}_{\lambda} )$ as a morphism of $L$-modules, since $\mathscr{L}_{\lambda}$ and $i^*_{L,z} \mathscr{L}_{\lambda}$ carry compatible $L$-equivariant structures. Since the extremal weights (e.g. the lowest weight) of $H^0(L/B_L, i^*_{L,z} \mathscr{L}_{\lambda})$ are in the image of the restriction map, we have a nonzero map of $L$-modules. Moreover, by the Borel-Weil theorem, $H^0(L/B_L, i^*_{L,z} \mathscr{L}_{\lambda})$ is an irreducible $L$-module, thus the restriction map must be surjective. 
\end{proof}

We thus have the following commutative diagram 

\begin{center}
\begin{tikzcd}
H^0(G/B, \mathscr{L}_{\lambda}) \arrow[r] \arrow[d] & H^0(L/B_L, i^*_{L,z} \mathscr{L}_{\lambda}) \arrow[d] \\
H^0(X_{yz}, \mathscr{L}_{\lambda}) \arrow[r] & H^0(X_{y}^L, i^*_{L,z} \mathscr{L}_{\lambda} ).\\
\end{tikzcd}
\end{center}

\begin{corollary} \label{Surj}
All maps in the above commutative diagram are surjections.
\end{corollary}

\begin{proof}
That the vertical arrows are surjections is a classical result on Schubert varieties. The top arrow is a surjection by the previous lemma, and this forces the bottom arrow to be a surjection as well.
\end{proof}

\subsection{Isomorphisms of T-equivariants}
Let $M$ be a representation of $T$. We write $M^T$ for the $T$-invariants. More generally, for $\xi$ a character of $T$, we write $M^{T, \xi}$ for the $\xi$-isotypic component of $M$. The following lemma is elementary.

\begin{lemma} \label{InvExact}
The functor $M \mapsto M^T$ is exact for $M$ in the category of rational $T$- modules. The functor $M \mapsto M^{T, \xi}$ is also exact for any character $\xi$; taking isotypic components is also exact on the category of rational $T$-modules.
\end{lemma}





In order to compare weight spaces of $H^0(X_{yz}, \mathscr{L}_{\lambda})$ with those in $H^0(X^L_{y}, i^*_{L,z} \mathscr{L}_{\lambda})$, using Lemma \ref{InvExact} and the surjection of Corollary \ref{Surj} we consider, for judiciously chosen $\xi$, the short exact sequence 
$$
0 \to H^0(X_{yz}, \mathscr{I}(Y^L_{yz}, X_{yz})\otimes \mathscr{L}_\lambda)^{T, \xi} \to H^0(X_{yz}, \mathscr{L}_\lambda)^{T, \xi} \to H^0(X^L_{y}, i^*_{L,z}\mathscr{L}_\lambda)^{T,\xi} \to 0
$$
where $\mathscr{I}(Y^L_{yz}, X_{yz})$ is the ideal sheaf of $i_{L,z}(X^L_y) = Y^L_{yz}$ in $X_{yz}$. As a preliminary step, we consider a further restriction of these global sections to the dense open cells of the corresponding Schubert varieties to more easily describe the $T$-weights which appear. We collect the following standard lemmas; for reasons which will become clear, for the remainder of this subsection we denote $q:=yz \in W$.

\begin{lemma}\label{affinespace}
We have $T$-equivariant isomorphisms $BqB/B \simeq \mathrm{Spec} \mathbb{C}[x_{-\beta_1} \dots x_{-\beta_{\ell(q)}}]$, where $\beta_i \in \Phi^+$ satisfy $q^{-1}.\beta_i \in \Phi^{-}$, e.g. $\beta_i \in R(q^{-1}).$
\end{lemma}

\begin{proof}
We have from \cite{Springer}*{Lemma 8.3.6} that $BqB/B \simeq \mathbb{A}^{\ell(q)}$; all that remains is to investigate the $T$-weights. Moreover by \cite{Springer}*{Lemma 8.3.5} we have a $T$-equivariant isomorphism $BqB/B \simeq \Pi_{\alpha \in R(q^{-1})} U_{\alpha}$. We have a $T$-equivariant isomorphism $U_{\alpha} \simeq \mathrm{Spec} \mathbb{C}[x_{-\alpha}]$, so we obtain
\begin{align*}
 BqB/B  & \simeq \Pi_{\beta \in R(q^{-1})}U_{\beta} \\ 
 & \simeq \Pi_{\beta \in R(q^{-1})} \mathrm{Spec} \mathbb{C}[x_{-\beta}] \\ 
 &\simeq \mathrm{Spec}( \bigotimes_{\beta \in R(q^{-1})} \mathbb{C}[x_{-\beta}])\\ 
 & \simeq \mathrm{Spec} \mathbb{C}[x_{-\beta_1}, \dots x_{-\beta_{l(q)}}]_{\beta_i \in R(q^{-1})}. \\
\end{align*}
\end{proof}

\begin{lemma}
Let $S:= \{ \beta \in R(q^{-1})| \beta \notin \Phi^+_L\}$.Via the isomorphism $BqB/B \simeq \mathrm{Spec} \mathbb{C}[x_{-\beta_1} \dots x_{\beta_{\ell(q)}}]$, the ideal \[I=I(B_LqB/B, BqB/B)\] whose vanishing locus is $B_LqB/B$ is generated by the $x_{-\beta}$ for $\beta \in S$.
\end{lemma}

\begin{proof}
The same considerations as the previous lemma lead us to the $T$-equivariant isomorphism $B_LqB/B \simeq \prod_{\beta \in R(q^{-1}) \cap \Phi_L^+} U_{\beta}$. This is an affine subspace of $\mathbb{A}^{\ell(q)}$ whose ideal is generated by the $x_{-\beta}$ such that $\beta \notin R(q^{-1} )\cap \Phi_L^+$, that is, $\beta \in S$.
\end{proof}

\begin{lemma} \label{cellweights}
We have a $T$-equivariant isomorphism \[H^0(BqB/B, \mathscr{L}_{\lambda}) \simeq \mathbb{C}[x_{-\beta_1}, \dots x_{-\beta_{\ell(q)}}] \otimes \mathbb{C}_{-q \lambda}.\]
\end{lemma}

\begin{proof}
Since $Pic(\mathbb{A}^n)=0$, $\mathscr{L}_{\lambda}|_{BqB/B}$ is isomorphic to $\mathscr{O}_{\mathbb{A}^{\ell(q)}}$, and what remains is to keep track of the $T$-weights. Since we already have a description of the $T$-module $H^0(BqB/B, \mathscr{O}_{BqB/B})$ from Lemma \ref{affinespace}, it suffices to understand the $T$-weight of $1 \in H^0(BqB/B, \mathscr{L}_{\lambda})$. Recall that $\mathscr{L}_{\lambda}= Sections(G \times_B \mathbb{C}_{-\lambda})$. From \cite{ChrissGinzburg}*{6.1.11}  we have $p:G \rightarrow G/B$ and for $U$ an open subvariety of $G/B$, a section $s \in H^0(U, \mathscr{L}_{\lambda})$ is a regular $\mathbb{C}$-valued function $\tilde{s}$ on $p^{-1}(U)$ such that \[\tilde{s}(gb)=\lambda(b) \cdot \tilde{s}(g).\] Let $\prescript{q}{}{U}$ denote the unipotent radical of $\prescript{q}{}{B}$, so that $\prescript{q}{}{U}qB$ is dense in $G$. Let $\tilde{\mathbf{1}}$ be the pullback $p^*(1)$, this is the unique (up to constants) function on $\prescript{q}{}{U}qB$ satisfying \[\tilde{\mathbf{1}}(uqb)=\lambda(b) \tilde{\mathbf{1}}(q).\] for $u \in \prescript{q}{}{U}$. A computation nearly identical to that of \cite{ChrissGinzburg}*{Lemma 6.1.15} yields 
 \[(t \cdot \tilde{\mathbf{1}})(uq)=\tilde{\mathbf{1}}(t^{-1}uq)=\tilde{\mathbf{1}}(t^{-1}q)=\lambda(q^{-1}t^{-1}q) \tilde{\mathbf{1}}(q)=-(q. \lambda)(t) \cdot \tilde{\mathbf{1}}(uq).\] 
 Thus the $T$-weight of $1 \in H^0(\prescript{q}{}{U}qB/B, \mathscr{L}_{\lambda})$ is $-q \lambda$. It follows that the weight of $1 \in H^0(BqB/B, \mathscr{L}_{\lambda})$ is $-q\lambda$; this is just restriction of functions to affine subschemes.
\end{proof}

\begin{lemma} \label{OpenCellRestriction}

We have $T$-equivariant inclusions
 \[0 \rightarrow H^0(X_q, \mathscr{L}_{\lambda}) \rightarrow H^0(BqB/B, \mathscr{L}_{\lambda}),\] 

  \[0 \rightarrow H^0(Y_q^{L}, \mathscr{L}_{\lambda}) \rightarrow H^0(B_LqB/B, \mathscr{L}_{\lambda}),\]
  and 
  \[0 \rightarrow H^0(X_q, \mathscr{I}(Y_q^L, X_q) \otimes \mathscr{L}_{\lambda}) \rightarrow H^0(BqB/B, I \otimes \mathscr{L}_{\lambda} )\]
\end{lemma}
\begin{proof}
We have a filtration of $X_q$ by closed subschemes $X_v$, $v \leq q$, the Schubert cells $BvB/B$ are affine, and the inclusions $i_v: BvB/B \rightarrow X_q$ are affine morphisms. These conditions ensure that the cohomology of the global Grothendieck-Cousin complex $\mathrm{Cousin}^{\bullet}_{\{X_v\}}(\mathscr{F})$ is isomorphic to $H^*(X_q, \mathscr{F})$, for $\mathscr{F} \in Coh(X_q)$, see for instance \cite{Kempf} or the Appendix of \cite{KumarBook}. In particular we have \[0 \rightarrow H^0(X_q, \mathscr{L}_{\lambda}) \rightarrow \mathrm{Cousin}_{\{X_v\}}^{0}(\mathscr{L}_{\lambda})=H^0(BqB/B, \mathscr{L}_{\lambda}).\] Mutatis mutandis for $H^0(Y_q^{L}, \mathscr{L}_{\lambda})$, $H^0(X_q, \mathscr{I}(Y_q^L, X_q) \otimes \mathscr{L}_{\lambda})$. 
\end{proof}

We summarize the above discussion with the following commutative diagram of rational $T$-representations. For display reasons we write $A= \mathbb{C}[x_{-\beta_1}, \dots x_{-\beta_{l(q)}}]$ where $\beta_i \in R(q^{-1})$.

\begin{equation} \label{idealses}
\begin{tikzcd}
& 0 \arrow[d] &0 \arrow[d] &0 \arrow[d] & \\
0 \arrow[r] &  H^0(X_q,\mathscr{I}(Y_q^L, X_q) \otimes \mathscr{L}_{\lambda}) \arrow[r] \arrow[d] &H^0(X_q, \mathscr{L}_{\lambda})\arrow[r] \arrow[d] & H^0(Y_q^L, \mathscr{L}_{\lambda}) \arrow[d] \arrow[r]& 0\\
 0 \arrow[r]& H^0(BqB/B, I \otimes \mathscr{L}_{\lambda}) \arrow[r] \arrow[d, equals] & H^0(BqB/B, \mathscr{L}_{\lambda}) \arrow[r] \arrow[d, equals] & H^0(B_LqB/B, \mathscr{L}_{\lambda}) \arrow[r] \arrow[d, equals] & 0 \\
 0 \arrow[r] & (\{x_{-\beta}\}_{\beta \in S}) \otimes \mathbb{C}_{-q \lambda} \arrow[r] & A \otimes \mathbb{C}_{-q \lambda} \arrow[r] & \arrow[r] A/(\{x_{\beta}\}_{\beta \in S}) \otimes \mathbb{C}_{-q \lambda} \arrow[r] & 0.\\ 
\end{tikzcd}
\end{equation}



 
 \subsection{Specialization to Theorem \ref{PreciseTheorem}} 
 We finally specialize the results of this section to their application in proving Theorem \ref{PreciseTheorem}(1). Our Levi $L$ is as before $L=\dot{v}L_\eta \dot{v}^{-1}$ for $L_\eta$ the Levi of the standard parabolic associated to the dominant coweight $\eta$ and $v \in W^\eta$. Retaining the notation as in Proposition \ref{LeviModuleInclusion}, we make the following associations: 
 $$
 z=v(\pi^\eta(w^{-1} \ast v))^{-1}, \ \ y =w_L = v(\pi_\eta(w^{-1} \ast v))^{-1} v^{-1}.
 $$
 This $z$ satisfies $\dot{z}^{-1}B_L \dot{z} \subseteq B$, as 
 $$
 \dot{z}^{-1} B_L \dot{z} = \pi^{\eta}(w^{-1} \ast v) v^{-1} \left( v B_\eta v^{-1} \right) v (\pi^\eta(w^{-1} \ast v))^{-1} = \pi^{\eta}(w^{-1} \ast v) B_\eta \pi^{\eta}(w^{-1} \ast v)^{-1} \subseteq B
 $$
 since $\pi^{\eta}(w^{-1} \ast v) \in W^\eta$. Further, 
 $$
 yz= v(\pi_\eta(w^{-1} \ast v))^{-1} v^{-1} \cdot v(\pi^\eta(w^{-1} \ast v))^{-1} = v(w^{-1} \ast v)^{-1} = v(v^{-1} q) = q
 $$
 where as before $q \leq w$ is the unique element such that $w^{-1} \ast v = q^{-1} v$. Thus, we can apply all of the results of this section to the fixed inclusion 
 $$
 i: X^L_{w_L} \hookrightarrow X_q.
 $$
 Finally, by Proposition \ref{pullback} we see that $i^\ast \mathscr{L}_\lambda \cong \mathscr{L}_{z\lambda} =\mathscr{L}_{\lambda_L}$. With this in hand, we aim to prove the following theorem, which is a reformulation of Theorem \ref{PreciseTheorem}(1) into the language of this section. 
 
 \begin{theorem} \label{RothFinale}
 Let $\mu $ be a weight of $V^w_\lambda$ lying on the face $\mathcal{F}(v,\eta)$ of the Demazure weight polytope $P^w_\lambda$. Then 
 $$
 H^0(X^L_{w_L}, \mathscr{L}_{\lambda_L})^{T, -\mu} \cong H^0(X_q, \mathscr{L}_\lambda)^{T, -\mu}.
 $$
 \end{theorem}
 
 Before proving this theorem, we give the final preparatory technical lemma which is relevant to this specific setup.

 \begin{lemma}\label{inversionsets}
Suppose that $\beta \in R(q^{-1})$. Then $v^{-1}\beta \succ 0$.
\end{lemma}

\begin{proof}
First, note that $q^{-1} \ast v = q^{-1}v$, since $q^{-1} \ast v = \mathrm{max}\{uv | u \leq q^{-1}\}$ by definition and we know that $w^{-1} \ast v = q^{-1} v$ and $q \leq w$. Since by assumption $\beta \in R(q^{-1})$, we have $q^{-1}s_\beta \le q^{-1}$. Thus by the definition of the Demazure product, necessarily $$q^{-1}s_\beta v \le q^{-1}v.$$  
We rewrite this as
$$
q^{-1}v(v^{-1}s_\beta v) \le q^{-1}v.
$$
But this means that the root $q^{-1}v \left(v^{-1}\beta \right)$ has opposite sign from the root $v^{-1}\beta$. Since $q^{-1}\beta \prec 0$, the result follows.
\end{proof}

\begin{lemma} \label{VanishLemma}
Let $\nu \in \mathbb{Z} \Phi_L$. We then have $((\{x_{-\beta}\}_{\beta \in S}) \otimes \mathbb{C}_{-q \lambda})^{T, -q \lambda + \mathbb{Z} \Phi_L}=0$. We thus have isomorphisms of weight spaces \[H^0(BqB, \mathscr{L}_{\lambda})^{T, -q \lambda + \mathbb{Z}\Phi_L} \simeq H^0(B_LqB/B, \mathscr{L}_{\lambda})^{T, -q \lambda+ \mathbb{Z} \Phi_L}.\]
\end{lemma}

\begin{proof}
A nonzero weight vector in $(\{x_{-\beta}\}_{\beta \in S}) \otimes \mathbb{C}_{-q \lambda})$ has weight $-q \lambda- \sum_{\beta \in R(q^{-1})} a_{\beta} \beta$, with $a_{\beta} \in \mathbb{Z}^{\geq 0}$ such that at least some $\beta \in S$ has $a_{\beta} \neq 0$. So we must show that weights $-\sum_{\beta \in R(q^{-1})} a_{\beta} \beta$ such that $a_{\beta} \neq 0$ for some $\beta \in S$ are not in $\mathbb{Z}\Phi_L$. After acting by $v^{-1}$, $\Phi_L$ becomes the root system of the standard Levi $L_{\eta}$. By \ref{inversionsets}, acting by $v^{-1}$ transforms $\beta \in S$ into positive roots, which are not roots of the standard Levi $L_{\eta}$ by definition of $S$. In particular we have $\langle \eta, v^{-1} \sum_{\beta \in R(q^{-1})} a_{\beta} \beta \rangle >0$, as long as $a_\beta >0$ for some $\beta \in S$, whereas $\langle \eta, \mathbb{Z}\Phi_{L_{\eta}} \rangle =0$. 
\end{proof}

\begin{proof}[Proof of Theorem \ref{RothFinale}]
Consider the short exact sequence 
$$
0 \to H^0(X_{q}, \mathscr{I}(Y^L_{q}, X_{q})\otimes \mathscr{L}_\lambda)^{T, -\mu} \to H^0(X_{q}, \mathscr{L}_\lambda)^{T, -\mu} \to H^0(X^L_{w_L}, \mathscr{L}_{\lambda_L})^{T,-\mu} \to 0
$$
as in the discussion following Lemma \ref{InvExact}. We will show that 
$$
H^0(X_{q}, \mathscr{I}(Y^L_{q}, X_{q})\otimes \mathscr{L}_\lambda)^{T, -\mu} = 0.
$$
By Lemma \ref{OpenCellRestriction}, we have an injection 
$$
0 \rightarrow H^0(X_{q}, \mathscr{I}(Y^L_{q}, X_{q})\otimes \mathscr{L}_\lambda)^{T, -\mu}  \rightarrow H^0(BqB/B, I \otimes \mathscr{L}_\lambda )^{T, -\mu} .
$$
As $-\mu \in -q\lambda +\mathbb{Z}\Phi_L$, by Lemma \ref{VanishLemma}, the desired vanishing follows, and thus 
$$
 H^0(X^L_{w_L}, \mathscr{L}_{\lambda_L})^{T, -\mu} \cong H^0(X_q, \mathscr{L}_\lambda)^{T, -\mu}.
$$

\end{proof}

\section{Proof of Theorem \ref{PreciseTheorem}(2)}

We now construct a proof of the second part of Theorem \ref{PreciseTheorem}, to compare the multiplicities $V_\lambda^w(\mu)$ and $V_\lambda^q(\mu)$ for $\mu \in \mathcal{F}(v, \eta)$ and the specific $q \leq w$ as constructed. Note that in many cases, we can have that $q=w$ in which case this second equality is unnecessary. As a special case of interest, this happens when $v=e \in W^\eta$, as then $w^{-1} \ast v = w^{-1}$. Thus, for the remainder of this section, we assume that $e \neq v \in W^\eta$, so that $l(v) \geq 1$. 

To this end, fix for the remainder of this section $s_i$ a simple reflection such that $s_iv \le v$; in particular, $v^{-1} \alpha_i\prec 0$. We record a few preparatory lemmas. 

\begin{lemma} \label{ConvexLemma1}
$\langle v^{-1}\alpha_i, \eta\rangle < 0$.
\end{lemma}

\begin{proof}
Since $v^{-1}\alpha_i\prec 0$ and $\eta$ is dominant, we must have $\langle v^{-1} \alpha_i, \eta \rangle \le 0$. 

If we had equality, then setting $\gamma := v^{-1}\alpha_i$ would give that $\gamma$ belongs to the sub-root system $\Phi_\eta$ of roots orthogonal to $\eta$. But the root $\gamma$ is a negative root, so $v(-\gamma)$ would be a positive root, as $v \in W^\eta$. But of course $v(-\gamma)=-\alpha_i$ is not a positive root. 
\end{proof}


\begin{lemma} \label{ConvexLemma2}
If $\mu+ k\alpha_i$ belongs to $P_\lambda^w$, then $k\le0$. 
\end{lemma}

\begin{proof}
If $\mu+k\alpha_i \in P^w_\lambda$, then we must have $\langle \lambda, (w^{-1}*v) \eta \rangle \le \langle \mu + k\alpha , v\eta \rangle$. Since $\mu \in \mathcal{F}(v, \eta)$, $\langle \lambda, (w^{-1} \ast v) \eta \rangle =\langle \mu, v\eta \rangle$, so that 
$$
0 \le k \langle v^{-1}\alpha_i, \eta \rangle.
$$
Lemma \ref{ConvexLemma1} now yields the result. 
\end{proof}

For this same $k$, we also have the following lower bound, which follows from \cite{BJK2}*{Lemma 4.5} and \cite{BJK2}*{Corollary 4.6}.

\begin{lemma} \label{ConvexLemma3}
If $\mu+k\alpha_i \in P^w_\lambda$, then $k \ge - \langle \mu, \alpha_i^\vee \rangle$. 
\end{lemma}

\begin{proof}
First we argue that $\langle \mu, \alpha_i^\vee \rangle\ge 0$. Indeed, by \cite{BJK2}*{Lemma 4.5} we can write 
$$
\mu = \sum_{
\tiny \begin{array}{c}
u\le w \\
u^{-1}\alpha_i \succ 0
\end{array}
} b_u u\lambda,
$$
where each $b_u\ge 0$. Then we have 
$$
\langle b_u u\lambda, \alpha_i^\vee\rangle \ge 0
$$
on account of the positivity of $u^{-1}\alpha_i^\vee$, the dominance of $\lambda$, and $b_u\ge 0$. Thus, $\langle \mu, \alpha_i^\vee \rangle \geq 0$. 

It is now equivalent to show that $s_i\mu \preceq \mu + k\alpha_i \preceq \mu$ in dominance order. If $s_i P_\lambda^w = P_\lambda^w$ then this is trivial, since the $\alpha_i$-string through $\mu$ contained in the polytope is preserved under $s_i$, and $s_i\mu \in P^w_\lambda$ must be the ``other'' end of this string opposite to $\mu$. 

Otherwise, we let $\mu''$ to be the other end of this string--that is, $\mu'' = \mu+k''\alpha_i$ with $k''$ minimal such that $\mu'' \in P^w_\lambda$. Then we have in this case that $P^w_\lambda \subseteq P^{s_iw}_\lambda$, and by \cite{BJK2}*{Corollary 4.6} $\mu$ is still the maximal endpoint of the $\alpha_i$-string inside $P^{s_iw}_\lambda$,  we have $s_i\mu \in P^{s_iw}_\lambda$ and thus $s_i \mu \preceq \mu'' \preceq \mu$ by the previous case. 
\end{proof}

\begin{proposition} \label{MultProp}
The multiplicity of $\mu$ in $\mathrm{ch}(V_\lambda^w)$ coincides with the multiplicity of $s_i\mu$ in $D_{i}\left(\mathrm{ch}(V_\lambda^w)\right)=\mathrm{ch}(V_\lambda^{s_i \ast w})$. 
\end{proposition}

\begin{proof}
If $s_i w < w$, then this is trivial since $D_{i}\left(\mathrm{ch}(V_\lambda^w)\right)=\mathrm{ch}(V_\lambda^w)$ and $s_i\left(\mathrm{ch}(V_\lambda^w)\right)=\mathrm{ch}(V_\lambda^w)$ by the $\mathfrak{sl}_{2, \alpha_i}$ module structure. 

Otherwise, we have $w < s_i w$. Write $\mathrm{ch}(V_\lambda^w) = \sum c_\nu e^\nu$. As in the proof of Lemma \ref{ConvexLemma3}, note that $\langle \mu, \alpha_i^\vee \rangle \geq 0$. By definition, the action of the Demazure operator $D_i$ on a term $e^\nu$ is given by 

  \[
	D_i(e^{\nu}) = \frac{1-e^{-\alpha_i}s_i}{1-e^{-\alpha_i}}e^\nu=  \begin{cases}
	e^{\nu}+e^{\nu-\alpha_i} \dots + e^{s_i \nu}, & \text{for } \langle \nu, \alpha_i^{\vee} \rangle \geq 0, \\
	0, & \text{for } \langle \nu, \alpha_i^{\vee} \rangle =-1,\\
	-(e^{\nu+ \alpha_i}+ \dots + e^{s_i \nu-\alpha_i}), & \text{for } \langle \nu, \alpha_i^{\vee} \rangle <-1.
	\end{cases} 
	\]

Consider the string $\mathcal{S} = (\mu + \Z \alpha_i) \cap P_\lambda^w$. By the Lemmas \ref{ConvexLemma2} and \ref{ConvexLemma3}, its maximal element is $\mu$, and its lowest element is $\mu - k_0 \alpha_i$ for some $k_0 \ge -\langle \mu, \alpha_i^\vee\rangle$. Let $\mathcal{T} = (\mu + \Z \alpha_i) \cap P_\lambda^{s_iw}$. Note that $\mathcal{T}$ is the set $\{s_i\mu, s_i\mu + \alpha_i, \hdots, \mu - \alpha_i, \mu\}$. It is straightforward to check that the result of the action of $D_{i}$ as above on the part of the character supported on $\mathcal{S}$ satisfies that only $D_i(e^\mu)$ can produce any terms of the form $e^{s_i \mu}$ and there are no further cancellations. 


From this we see that the coefficient of $s_i\mu$ (and also still of $\mu$) in $\mathrm{ch}(V_\lambda^{s_iw})$ is $c_\mu$, as desired. 
\end{proof}

\begin{lemma} \label{InductFace}
In $P^{s_i \ast w}_\lambda$, $s_i\mu$ lies on the face $\mathcal{F}(s_iv, \eta)$.
\end{lemma}
\begin{proof}
We need just check that 
$$
\langle \lambda, ((s_i*w)^{-1}*(s_iv))\eta \rangle = \langle s_i\mu, (s_iv)\eta \rangle
$$
The Weyl group element on the left side simplifies to $w^{-1}*s_i*(s_iv) = w^{-1}*v$. The right simplifies to $\langle \mu, v\eta\rangle$; this is just the original equality for $\mu \in \mathcal{F}(v, \eta)$ which holds by assumption. 
\end{proof}

\begin{theorem} \label{ConnectingTheorem}
Let $\mu \in \mathcal{F}(v, \eta)$ in $P_\lambda^w$. Then the multiplicity of $\mu$ in $\mathrm{ch}(V_\lambda^w)$ coincides with the multiplicity of $v^{-1}\mu$ in $\mathrm{ch}(V_\lambda^{v^{-1}*w})$. Moreover, $v^{-1}\mu$ lies on the face $\mathcal{F}(e, \eta)$ of $P_\lambda^{v^{-1} \ast w}$. 
\end{theorem}

\begin{proof}
We proceed by induction on $\ell(v)$. The case $\ell(v)=1$ is given by Propsition \ref{MultProp}.

Write $v = s_{i}u$, so $s_i v=u \le v$. Again by Proposition \ref{MultProp} we have that the coefficient of $e^\mu$ in $\mathrm{ch}(V_\lambda^w)$ coincides with the weight of $e^{s_i\mu}$ in $\mathrm{ch}(V_\lambda^{s_i*w})$; and, $s_i\mu$ belongs to the face corresponding to $\mathcal{F}(u, \eta)$ on $P_\lambda^{s_i*w}$ by Lemma \ref{InductFace}.

Applying the induction hypothesis to $s_i \mu \in \mathcal{F}(u, \eta)$ in $P_\lambda^{s_i \ast w}$, we get that the multiplicity of $s_i\mu$ in $\mathrm{ch}(V_\lambda^{s_i*w})$ coincides with the multiplicity of $u^{-1}s_i\mu = v^{-1} \mu$ in $\mathrm{ch}(V_\lambda^{u^{-1}*s_i*w})= \mathrm{ch}(V_\lambda^{v^{-1}\ast w})$, and $u^{-1}s_i\mu = v^{-1}\mu$ belongs to the face $\mathcal{F}(e, \eta)$ on $P_\lambda^{u^{-1}*s_i *w}=P_\lambda^{v^{-1} \ast w}$. 
\end{proof}


\tdplotsetmaincoords{70}{325}
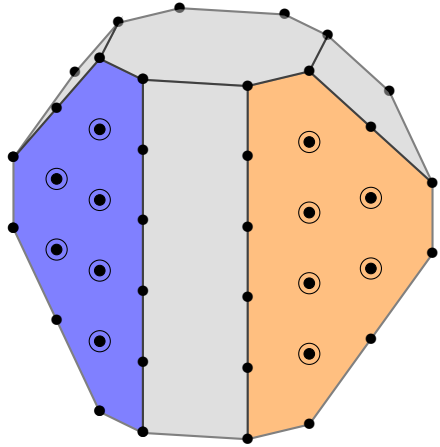
\begin{figure}
\begin{center}
\begin{tikzpicture}[tdplot_main_coords]

\coordinate(s3s2s3s1s2s3s2) at (-3/2, -1/2, -5/2);
\coordinate(s2s3s1s2s3s2) at (-3/2, -1/2, 5/2);
\coordinate(s3s1s2s3s1s2) at (-3/2, 1/2, -5/2);
\coordinate(s3s2s3s1s2s3) at (-1/2, -3/2, -5/2); 
\coordinate(s3s1s2s3s2) at (-3/2, 5/2, -1/2);
\coordinate(s1s2s3s1s2) at (-3/2, 1/2, 5/2);
\coordinate(s2s3s1s2s3) at (-1/2, -3/2, 5/2); 
\coordinate(s3s2s3s1s2) at (1/2, -3/2, -5/2); 
\coordinate(s3s1s2s3s1) at (-1/2, 3/2, -5/2);
\coordinate(s1s2s3s2) at (-3/2, 5/2, 1/2);
\coordinate(s3s2s3s2) at (5/2, -3/2, -1/2); 
\coordinate(s3s1s2s3) at (-1/2, 5/2, -3/2);
\coordinate(s2s3s1s2) at (1/2, -3/2, 5/2); 
\coordinate(s1s2s3s1) at (-1/2, 3/2, 5/2);
\coordinate(s3s1s2s1) at (1/2, 3/2, -5/2);
\coordinate(s3s2s3s1) at (3/2, -1/2, -5/2);
\coordinate(s2s3s2) at (5/2, -3/2, 1/2); 
\coordinate(s3s1s2) at (1/2, 5/2, -3/2);
\coordinate(s1s2s3) at (-1/2, 5/2, 3/2);
\coordinate(s3s2s3) at (5/2, -1/2, -3/2);
\coordinate(s1s2s1) at (1/2, 3/2, 5/2);
\coordinate(s2s3s1) at (3/2, -1/2, 5/2);
\coordinate(s3s2s1) at (3/2, 1/2, -5/2);
\coordinate(s3s2) at (5/2, 1/2, -3/2);
\coordinate(s2s3) at (5/2, -1/2, 3/2);
\coordinate(s1s2) at (1/2, 5/2, 3/2);
\coordinate(s3s1) at (3/2, 5/2, -1/2);
\coordinate(s2s1) at (3/2, 1/2, 5/2);
\coordinate(s2) at (5/2, 1/2, 3/2);
\coordinate(s3) at (5/2, 3/2, -1/2);
\coordinate(s1) at (3/2, 5/2, 1/2);
\coordinate(e) at (5/2, 3/2, 1/2);


\node at (s3s2s3s1s2s3s2) {\small $\bullet$};
\node at (s2s3s1s2s3s2) {\small $\bullet$};
\node at (s3s1s2s3s1s2) {\small $\bullet$};
\node at (s3s1s2s3s2) {\small $\bullet$};
\node at (s1s2s3s1s2) {\small $\bullet$};
\node at (s1s2s3s2) {\small $\bullet$};
\node at (s1s2s3s1) {\small $\bullet$};
\node at (s1s2s3) {\small $\bullet$};
\node at (s1s2s1) {\small $\bullet$};
\node at (s2s3s1) {\small $\bullet$};
\node at (s2s3) {\small $\bullet$};
\node at (s2s1) {\small $\bullet$};


\draw[thick,fill=orange,opacity=0.5] (s3s2s3s1s2s3) -- (s3s2s3s1s2) -- (s3s2s3s2) -- (s2s3s2)
  -- (s2s3s1s2) -- (s2s3s1s2s3) -- (s3s2s3s1s2s3); 
  

\draw[thick,fill=blue,opacity=0.5] (s3s2s3s1s2s3s2) -- (s2s3s1s2s3s2) -- (s1s2s3s1s2) -- (s1s2s3s2) -- (s3s1s2s3s2) -- (s3s1s2s3s1s2) -- (s3s2s3s1s2s3s2);
  

\draw[thick,fill=lightgray,opacity=0.5] (s2s3s1s2) -- (s2s3s1) -- (s2s3) -- (s2s3s2) -- (s2s3s1s2);

\draw[thick,,fill=lightgray,opacity=0.5] (s2s3s1s2s3s2) -- (s1s2s3s1s2) -- (s1s2s3s1) -- (s1s2s1) -- (s2s1) -- (s2s3s1) -- (s2s3s1s2) -- (s2s3s1s2s3) -- (s2s3s1s2s3s2);

\draw[thick,fill=lightgray,opacity=0.5] (s1s2s3s1s2) -- (s1s2s3s2) -- (s1s2s3) -- (s1s2s3s1) -- (s1s2s3s1s2);

\draw[thick,fill=lightgray,opacity=0.5] (s3s2s3s1s2s3s2) -- (s2s3s1s2s3s2) -- (s2s3s1s2s3) -- (s3s2s3s1s2s3) -- (s3s2s3s1s2s3s2);




\coordinate(w1) at (-1/2, -3/2, 5/2);
\coordinate(w2) at (-1/2, -3/2, 3/2);
\coordinate(w3) at (-1/2, -3/2, 1/2);
\coordinate(w4) at (-1/2, -3/2, -1/2);
\coordinate(w5) at (-1/2, -3/2, -3/2);
\coordinate(w6) at (-1/2, -3/2, -5/2);
\coordinate(w7) at (1/2, -3/2, 5/2);  
\coordinate(w8) at (1/2, -3/2, 3/2);  
\coordinate(w9) at (1/2, -3/2, 1/2);  
\coordinate(w10) at (1/2, -3/2, -1/2);  
\coordinate(w11) at (1/2, -3/2, -3/2);  
\coordinate(w12) at (1/2, -3/2, -5/2);
\coordinate(w13) at (3/2, -3/2, 3/2);  
\coordinate(w14) at (3/2, -3/2, 1/2);  
\coordinate(w15) at (3/2, -3/2, -1/2);  
\coordinate(w16) at (3/2, -3/2, -3/2);
\coordinate(w17) at (5/2, -3/2, 1/2);
\coordinate(w18) at (5/2, -3/2, -1/2);

\node at (w1) {\small $\bullet$};
\node at (w2) {\small $\bullet$};
\node at (w3) {\small $\bullet$};
\node at (w4) {\small $\bullet$};
\node at (w5) {\small $\bullet$};
\node at (w6) {\small $\bullet$};
\node at (w7) {\small $\bullet$};
\node at (w12) {\small $\bullet$};
\node at (w13) {\small $\bullet$};
\node at (w16) {\small $\bullet$};
\node at (w17) {\small $\bullet$};
\node at (w18) {\small $\bullet$};

\filldraw (w8) circle[radius=2pt];
\draw (w8) circle[radius=4pt];

\filldraw (w9) circle[radius=2pt];
\draw (w9) circle[radius=4pt];

\filldraw (w10) circle[radius=2pt];
\draw (w10) circle[radius=4pt];

\filldraw (w11) circle[radius=2pt];
\draw (w11) circle[radius=4pt];

\filldraw (w14) circle[radius=2pt];
\draw (w14) circle[radius=4pt];

\filldraw (w15) circle[radius=2pt];
\draw (w15) circle[radius=4pt];


\coordinate(x1) at (-3/2, -1/2,  5/2);
\coordinate(x2) at (-3/2, -1/2,  3/2);
\coordinate(x3) at (-3/2, -1/2,  1/2);
\coordinate(x4) at (-3/2, -1/2,  -1/2);
\coordinate(x5) at (-3/2,  -1/2, -3/2);
\coordinate(x6) at (-3/2, -1/2,  -5/2);
\coordinate(x7) at (-3/2, 1/2,  5/2);  
\coordinate(x8) at (-3/2, 1/2,  3/2);  
\coordinate(x9) at (-3/2, 1/2,  1/2);  
\coordinate(x10) at (-3/2, 1/2,  -1/2);  
\coordinate(x11) at ( -3/2, 1/2, -3/2);  
\coordinate(x12) at (-3/2, 1/2,  -5/2);
\coordinate(x13) at (-3/2, 3/2,  3/2);  
\coordinate(x14) at (-3/2, 3/2,  1/2);  
\coordinate(x15) at (-3/2, 3/2,  -1/2);  
\coordinate(x16) at (-3/2, 3/2,  -3/2);
\coordinate(x17) at (-3/2,  5/2, 1/2);
\coordinate(x18) at (-3/2,  5/2, -1/2);

\node at (x1) {\small $\bullet$};
\node at (x2) {\small $\bullet$};
\node at (x3) {\small $\bullet$};
\node at (x4) {\small $\bullet$};
\node at (x5) {\small $\bullet$};
\node at (x6) {\small $\bullet$};
\node at (x7) {\small $\bullet$};
\node at (x12) {\small $\bullet$};
\node at (x13) {\small $\bullet$};
\node at (x16) {\small $\bullet$};
\node at (x17) {\small $\bullet$};
\node at (x18) {\small $\bullet$};

\filldraw (x8) circle[radius=2pt];
\draw (x8) circle[radius=4pt];

\filldraw (x9) circle[radius=2pt];
\draw (x9) circle[radius=4pt];

\filldraw (x10) circle[radius=2pt];
\draw (x10) circle[radius=4pt];

\filldraw (x11) circle[radius=2pt];
\draw (x11) circle[radius=4pt];

\filldraw  (x14) circle[radius=2pt];
\draw (x14) circle[radius=4pt];

\filldraw (x15) circle[radius=2pt];
\draw (x15) circle[radius=4pt];

\end{tikzpicture}
\caption{\label{example-figure-2} Demazure polytope for type $B_3$, highest weight $\lambda = \rho = \omega_1+\omega_2+\omega_3$ and $w =  s_1s_3s_2s_3s_1s_2s_3$. Note that $s_1*w = w$. The face corresponding to $\eta = x_1, v = s_1$ is highlighted in orange, while the face corresponding to $\eta = x_1, v = e$ is highlighted in blue. The two faces and their weight spaces are in bijection via $s_1$.  }
\end{center}
\end{figure}

\begin{example}
Continuing with the example in type $B_3$ with $\lambda = \rho$, $w = s_1s_3s_2s_3s_1s_2s_3$, $v = s_1$, and $\eta = x_1$, Figure \ref{example-figure-2} illustrates that the portion of the character of $V_\lambda^w$ on the face $\mathcal{F}(v,\eta)$ is identified, via multiplication by $v^{-1}$, with the portion of the character of $V_\lambda^{v^{-1}*w}$ on the face $\mathcal{F}(e,\eta)$. This illustrates Theorem \ref{ConnectingTheorem}. 

Although not shown in the diagram, the same is true for the portion of the character of $V_\lambda^q$ on that same face. Hence the multiplicities along the face are the same for $V_\lambda^w$ and $V_\lambda^q$, which is the final argument in the proof of Theorem \ref{PreciseTheorem}(2). 
\end{example}

The proof of the second part of Theorem \ref{PreciseTheorem} is now an easy observation following from Theorem \ref{ConnectingTheorem}.


\begin{proof}[Proof of Theorem \ref{PreciseTheorem}(2)]
Recall that $q$ is the unique Weyl group element in the Bruhat interval $[e,w]$ such that $w^{-1} \ast v = q^{-1}v$. Fix $\mu$ a weight on the shared face $\mathcal{F}(v,\eta)$ of $P_\lambda^q$ and $P_\lambda^w$. Since $v^{-1}*q = v^{-1}*w$ by this judicious choice of $q$, Theorem \ref{ConnectingTheorem} gives that the multiplicity of $v^{-1}\mu$ in $\mathrm{ch}(V_\lambda^{v^{-1} \ast w}) = \mathrm{ch}(V_\lambda^{v^{-1} \ast q})$ coincides respectively with the multiplicity of $\mu$ in $\mathrm{ch}(V_\lambda^w)$ and $\mathrm{ch}(V_\lambda^q)$; the latter multiplicities are thus equal, as desired. 
\end{proof}





\end{document}